\documentclass[review, 12pt]{elsarticle}

\usepackage{here} 
\usepackage{eurosym}
\usepackage{setspace}
\usepackage{ulem}
\usepackage[utf8]{inputenc}
\usepackage[margin=1in]{geometry}
\usepackage{geometry}

\usepackage{amsmath,amsfonts,amssymb,amsthm,latexsym,graphics,epsfig,url}
\usepackage{lineno,hyperref}
\usepackage{color}
\usepackage{multicol}
\usepackage{multirow}
\usepackage{tikz}
\usepackage[algo2e,ruled,vlined]{algorithm2e}
\usepackage{algorithm}
\usepackage{algpseudocode}
\usepackage[english]{babel}
\usepackage{epsfig}
\usepackage{graphicx}
\usepackage{bm}
\usepackage{fullpage}

\newcommand{\executeiffilenewer}[3]{%
\ifnum\pdfstrcmp{\pdffilemoddate{#1}}%
{\pdffilemoddate{#2}}>0%
{\immediate\write18{#3}}\fi%
}

\newtheorem{theo}{Theorem}[section]
\newtheorem{propo}[theo]{Proposition}
\newtheorem{lema}[theo]{Lemma}
\newtheorem{defi}[theo]{Definition}
\newtheorem{example}[theo]{Example}
\newtheorem{problem}[theo]{Problem}

\newtheorem{coro}[theo]{Corollary}

\newtheorem{obs}[theo]{Observation}

\def\matrix0{{\mbox {\boldmath $O$}}}

\def\vec0{\mbox{\bf 0}}

\def\Z{\mathbb{Z}}

\def\mod{\mathop{\rm mod }\nolimits}

\def\t{{^{\!\top\!\!\!}}}
\def\1{{\bf 1}}

\definecolor{red}{rgb}{1,0,0}

\definecolor{blue}{rgb}{0,0,1}

\begin{document}
\begin{frontmatter}

\title{On the status sequences of trees}

\author[tue,ghent]{Aida Abiad}
\ead{a.abiad.monge@tue.nl}
\author[rice]{Boris Brimkov}
\ead{boris.brimkov@sru.edu}
\author[maastricht,novosibirsk]{Alexander Grigoriev}
\ead{a.grigoriev@maastrichtuniversity.nl}

\address[tue]{Department of Mathematics and Computer Science,\\ Eindhoven University of Tehcnology, Eindhoven, The Netherlands}

\address[ghent]{Department of Mathematics: Analysis, Logic and Discrete Mathematics,\\ Ghent University, Ghent, Belgium}


\address[rice]{Department of Mathematics and Statistics, Slippery Rock University, Slippery Rock, PA, USA}

\address[novosibirsk]{Novosibirsk State University, Novosibirsk, Russia}

\begin{abstract}
The status of a vertex $v$ in a connected graph is the sum of the distances from $v$ to all other vertices. The status sequence of a connected graph is the list of the statuses of all the vertices of the graph. In this paper we investigate the status sequences of trees. Particularly, we show that it is NP-complete to decide whether there exists a tree that has a given sequence of integers as its status sequence. We also present some new results about trees whose status sequences are comprised of a few distinct numbers or many distinct numbers. In this direction, we show that any status injective tree is unique among trees. Finally, we investigate how orbit partitions and equitable partitions relate to the status sequence.
\end{abstract}

\end{frontmatter}


\section{Introduction}
Sequences associated with a graph, such as the degree sequence, spectrum, and status sequence, contain useful information about the graph's structure and give a compact representation of the graph without using vertex adjacencies. Extracting and analyzing the information contained in such sequences is a crucial issue in many problems, such as graph isomorphism. In this paper, we study the status sequences of trees, and answer several questions about status realizability, uniqueness, and their relation to various graph partitions.

The  graphs  considered in  this  paper  are  finite,  simple, loopless and connected.  Let $G=(V,E)$ be a graph. The {\it status value} (also called {\it transmission index}) of a vertex $v$ in $G$, denoted $s(v)$, is the sum of the distances between $v$ and all other vertices, i.e., $s(v)=\sum_{u\in V}d(v,u)$, where $d(v,u)$ is the shortest path distance between $v$ and $u$ in $G$. This concept was introduced by Harary in 1959~\cite{H1959}. The {\it status sequence} of $G$, denoted $\sigma(G)$, is the list of the status values of all vertices arranged in nondecreasing order. Status sequences of graphs have recently attracted considerable attention, see, e.g., \cite{LTS2012,P1997,SL2011,S2014}.

A connected graph is {\it status injective} if the status values of its vertices are all distinct. We denote by $k(G)$ the {\it number of different status values} in $\sigma(G)$. A graph is {\it transmission-regular} if $k(G)=1$, i.e., if the status values of all its vertices are equal. Transmission-regular graphs have been studied by several authors (see, e.g., \cite{GRWC2016,AP2016,IG2008,I2009}). A sequence $\sigma'$ of integers is {\it status realizable} if there exists a graph $G$ with $\sigma(G)=\sigma'$. Let $\mathcal{F}$ be a family of connected graphs and $G$ be a graph in $\mathcal{F}$. A graph $G$ is {\it status unique in} $\mathcal{F}$ if for any $H\in \mathcal{F}$, $\sigma(H)=\sigma(G)$ implies that $H\simeq G$, i.e., $H$ and $G$ are isomorphic. For example, paths are status unique in the family of all connected graphs, since a path of order $n$ is the only graph of order $n$ containing a vertex of status value $n(n-1)/2$.

In general, compared to the adjacency list or the adjacency matrix of a graph, there is some loss of information in the status sequence. For example, one cannot obtain the distance degree sequence of a graph from its status sequence \cite{P1997}. It is also well known that non-isomorphic graphs may have the same status sequence. Nevertheless, the status sequence contains important information about the graph, and the study of graphs with special status sequences has produced interesting results. For instance, it is known that spiders are status unique in trees \cite{SL2011}. In (\cite{HB1990}, p.185) the authors proposed the problem of finding status injective graphs. This problem was addressed by Pachter \cite{P1997}, who proved that for any graph $G$ there exists a status injective graph $H$ that contains $G$ as an induced subgraph.

Buckley and Harary \cite{HB2002} posed the following problem in a paper that discusses various problems concerning distance concepts in graphs.

\begin{problem}[{\sc Status sequence recognition}]\label{problem11BH}
Characterize status sequences, i.e, find a characterization that determines whether a given sequence of positive integers is the status sequence of a graph.
\end{problem}

Motivated by the above facts, in this paper we consider the following  question raised by Shang and Lin in \cite{SL2011}:
\begin{problem}\label{con1}
Are status injective trees status unique in all connected graphs?
\end{problem}

The above problem  was recently answered in the negative in \cite{QZ2019}, where the authors provide a construction of pairs of a tree and a nontree graph with the same status sequence. In this paper, we provide new results on the direction of the above problem for extremal trees. In particular, we show that status injective trees are status unique in trees.

Moreover, we also investigate the following special case of Problem \ref{problem11BH}:
\begin{problem}[{\sc Tree status recognition}]\label{problemstatusrealisable1}
Given a sequence $\sigma'$ of positive integers, does there exist a tree $T$ such that $\sigma(T)=\sigma'$?
\end{problem}
Recognition problems similar to \ref{problem11BH} and \ref{problemstatusrealisable1} have been studied for some other graph sequences. For example, Erd\"os and Gallai \cite{EG1960} and Hakimi \cite{H1962} gave conditions to determine whether a given sequence is the degree sequence of some graph. Sequences related to distances in a graph have also been studied, see, e.g., \cite{HB2002}. In those papers, the authors not only tackle the recognition problems, but also construct fast algorithms for finding graphs that {\it realize} given sequences. Following the same direction, we address the following question regarding status sequences:
\begin{problem}[{\sc Status realizability in trees}]\label{problemstatusrealisable2}
Given a sequence of integers $\sigma'$, either construct a tree $T$ such that $\sigma(T)=\sigma'$ or report that such a tree does not exist.
\end{problem}

We are also interested in the following problem related to the status uniqueness of a graph:
\begin{problem}
What are necessary and/or sufficient conditions for a set of vertices of a graph to have the same status?
\end{problem}

This paper is organized as follows. In Section \ref{section:complexity}, we prove that {\sc Tree status recognition} is NP-complete. In Section \ref{section_realizability_trees}, we present several polynomially solvable special cases of {\sc Status realizability in trees}, specifically for symmetric and asymmetric trees. In Section \ref{section_status_partitions}, we explore how various well-known graph partitions relate to the status sequences.

\section{Complexity of {\sc Tree status recognition}}\label{section:complexity}
Let the {\it depth} of a tree $T=(V,E)$ be the smallest number $k$ such that there exists a vertex $v\in V$ such that $d(v,u)\leq k$ for all $u\in V$. This is equivalent to saying that the diameter of the tree is at most $2k$. In this section, we first study the complexity of {\sc Tree status recognition} when the trees are restricted to have the depth of $3$. We refer to the latter problem as {\sc SRT-D3}.
\begin{theo}\label{thm:complexity}
{\sc SRT-D3} is NP-complete.
\end{theo}

\begin{proof}
SRT-D3 is clearly in NP, as the status sequence of a tree of depth 3 can be computed in linear time. We reduce the well-known strongly NP-complete problem {\sc 3-Partition} to SRT-D3. {\sc 3-Partition} (cf.~\cite{GareyJohnson79}) reads:\ Given a multiset $\mathcal{A}=\{a_1,a_2,\ldots,a_n\}$ of positive integers, does there exist a partition of $\mathcal{A}$ into triplets such that all triplets have the same sum?  Without loss of generality, assume $n$ is divisible by 3 and let $m=n/3$,\ $A=\sum_{i=1}^n a_i$ and $B=A/m$. Since {\sc 3-Partition} is strongly NP-complete, we may assume that the input of the problem is provided in unary encoding. Note that {\sc 3-Partition} remains strongly NP-complete even if $B/4<a_i<B/2$ for all $i\in \{1,2,\ldots,n\}$.

Given an instance $I$ of {\sc 3-Partition}, we construct an instance of {\sc SRT-D3} whose input sequence $\sigma'$ is comprised of the following integers:

\begin{itemize}
\item $3A+7m$ with multiplicity 1;
\item $4A-2B+11m-7$ with multiplicity $m$;
\item $5A-2B-2a_i+15m-8$ with multiplicity 1, for each $i\in\{1,2,\ldots,n\}$;
\item $6A-2B-2a_i+19m-9$ with multiplicity $a_i$, for each $i\in\{1,2,\ldots,n\}$.
\end{itemize}

Clearly, the size of the sequence $\sigma'$ is polynomial in the unary encoded size of $I$. We will show that $I$ is a yes-instance of {\sc 3-Partition} if and only if $\sigma'$ is a yes-instance of {\sc SRT-D3}. Suppose first that $I$ is a yes-instance of {\sc 3-Partition}, i.e., there is a partition of $\mathcal{A}$ into $m$ triplets such that the sum of integers in each triplet equals $B$. Consider the following tree $T$. Take a root vertex with exactly $m$ child-nodes which represent $m$ triplets.  For each child-node representing a triplet, create exactly three descendants representing the elements of the triplet. Finally, for each vertex representing an element $a_i$ of a triplet, create exactly $a_i$ descendants, which are the leaves of the tree. By construction the tree is of depth 3. See Figure \ref{fig:tree3partition} for an illustration.
\begin{center}
\begin{figure}[ht]
\centering
  \includegraphics[scale=0.3]{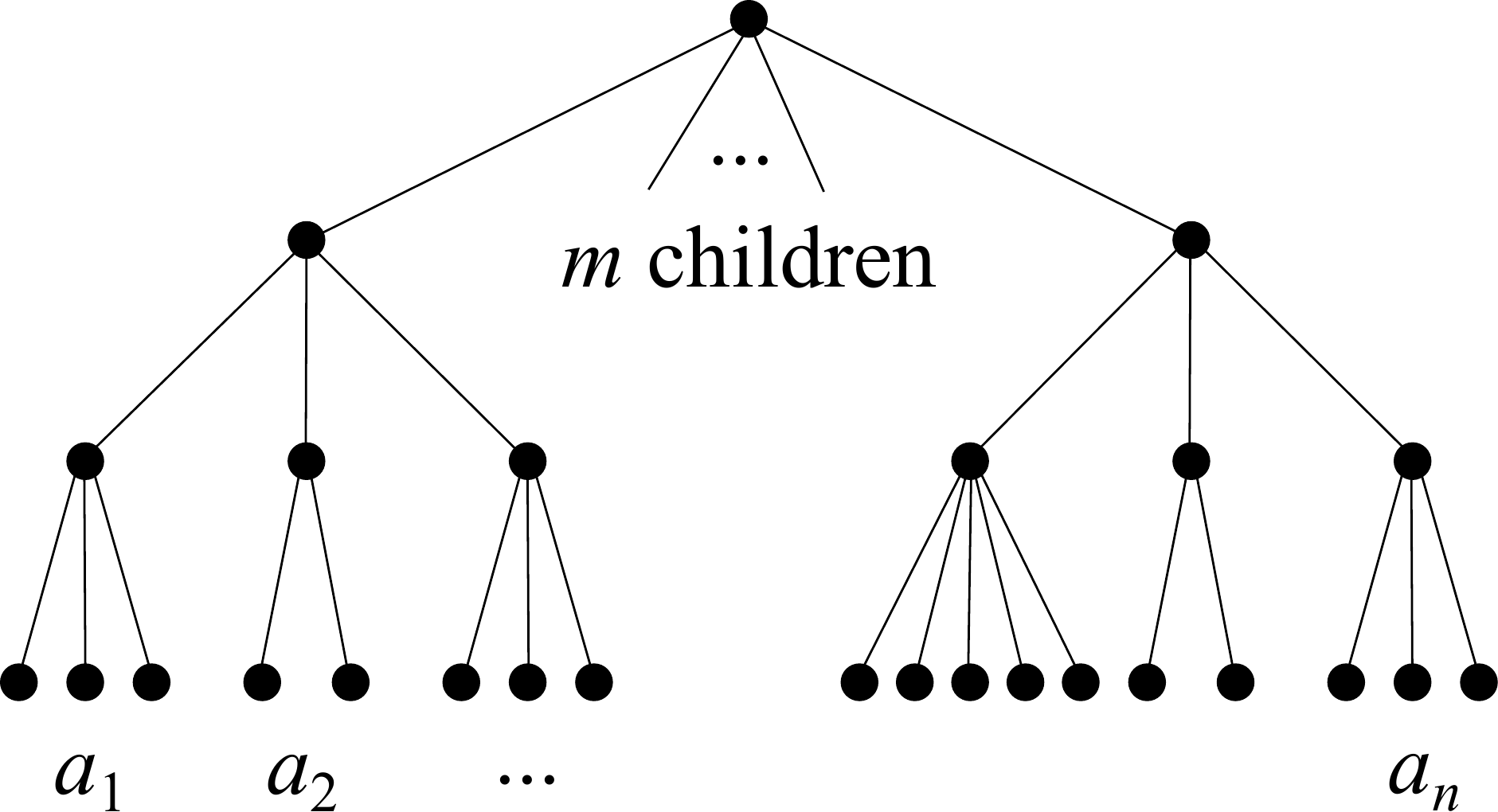}
    \caption{Gadget of the reduction from {\sc 3-Partition}}
  \label{fig:tree3partition}
\end{figure}
\end{center}

The status of the root equals $1\cdot m+2\cdot 3m+3\cdot \sum_i a_i=3A+7m$. The status of any vertex representing a triplet equals $1\cdot 4+2\cdot (B+m-1) +3\cdot 3(m-1)+ 4\cdot (m-1)B=4A-2B+11m-7$. There are $m$ such vertices. The status of a vertex representing an element $a_i$ equals $1\cdot (a_i+1)+2\cdot 3+3\cdot (B-a_i+m-1)+4\cdot 3(m-1)+5\cdot (m-1)B=5A-2B-2a_i+15m-8$. Finally, the status of a leaf adjacent to a vertex representing $a_i$ equals $1\cdot 1+ 2\cdot a_i+3\cdot 3+4\cdot (B-a_i+m-1)+ 5\cdot 3(m-1)+ 6\cdot (m-1)B=6A-2B-2a_i+19m-9$. Thus, tree $T$ of depth 3 realizes $\sigma'$.

Before tackling the opposite direction of the proof, we first recall some useful properties of the status values of the vertices in a tree. The {\it median} of a connected graph $G$ is the set of vertices of $G$ with the smallest status.

\begin{lema}[\cite{EJS1976}]\label{EJS1976}
If $v_1$ is a vertex in the median of a tree $T$, $v_1,v_2,\ldots ,v_r$ is a path in $T$, and $v_2$ is not a median of $T$, then $s(v_1)<s(v_2)<\cdots < s(v_r)$.
\end{lema}

\begin{lema}[\cite{SL2011}] \label{lema2.2SL2011}
Suppose that $v_1$ and $v_2$ are adjacent vertices in a tree $T=(V,E)$. Let $T_1$ and $T_2$ be two components of $T$ after deletion of the edge $(v_1,v_2)$, and let $v_1\in V(T_1)$ and $v_2\in V(T_2)$. Then, $s(v_1)-s(v_2)=|V(T_2)|-|V(T_1)|$.
\end{lema}

Now, suppose, $\sigma'$ is a yes-instance of SRT-D3 and a tree $T$ is its realization.  Without loss of generality, we may assume $A>3B+19m+9$, for otherwise we add to every element of the multiset an additive constant $3B+19m+9$ preserving the hardness of the {\sc 3-partition} case and repeat the arguments. For these sufficiently large values $A$, the vertex with status $3A+7m$ is the only median of $T$. Moreover, by Lemma~\ref{EJS1976} for all $i\in\{1,2,\ldots,n\}$, the vertices of status $6A-2B-2a_i+19m-9$ are the leaves of $T$. Furthermore, by Lemma~\ref{lema2.2SL2011} for all $i\in\{1,2,\ldots,n\}$, any vertex of status $6A-2B-2a_i+19m-9$ is adjacent to a vertex of status $5A-2B-2a_i+15m-8$ as the difference in status between a leaf and its ancestor equals the size of the tree minus 2, which in $T$ equals $A+4m-1$ . Next, we notice that for any $i\in\{1,2,\ldots,n\}$, the vertex of status $5A-2B-2a_i+15m-8$ cannot be directly adjacent to the median with status $3A+7m$ as the difference in statuses is too large contradicting Lemma~\ref{lema2.2SL2011}. This implies that for any $i\in\{1,2,\ldots,n\}$, the ancestor of a vertex of status $5A-2B-2a_i+15m-8$ can only be a vertex of status $4A-2B+11m-7$ and all vertices of the latter status are adjacent to the median. Since all vertices adjacent to the median have the same status $4A-2B+11m-7$, each of these vertices has exact the same total number of descendants, namely $B+3$.

Consider an edge between a vertex $v$ of status $5A-2B-2a_i+15m-8,\ i\in\{1,2,\ldots,n\}$, and its ancestor $u$ of status $4A-2B+11m-7$. By Lemma~\ref{lema2.2SL2011}, the number of leaves $x$ adjacent to $v$ equals the size of the tree after deletion of $v$ together with all its $x$ descendants minus the status difference between $v$ and $u$ and minus one for counting $v$ itself. The size of the tree after deletion of $v$ with all its descendants is $A+4m+1-x-1=A+4m-x$. The difference in status of $v$ and $u$ is $5A-2B-2a_i+15m-8-(4A-2B+11m-7)=A-2a_i+4m-1$. Therefore, $x=A+4m-x-(A-2a_i+4m-1)-1=2a_i-x$, yielding $x=a_i$.

Summarizing the findings, with necessity the following properties hold for $T$:
\begin{enumerate}
\item The median of $T$ is the vertex of status $3A+7m$;
\item There are $m$ vertices of status $4A-2B+11m-7$ adjacent to the median;
\item There are exactly $B+3$ descendants for every vertex of status $4A-2B+11m-7$;
\item Only vertices of status $5A-2B-2a_i+15m-8,\ i\in\{1,2,\ldots,n\}$, can be immediate successors/descendants of a vertex of status $4A-2B+11m-7$;
\item For every vertex $v_i$ of status $5A-2B-2a_i+15m-8,\ i\in\{1,2,\ldots,n\}$, there are exactly $a_i$ leaves of status $6A-2B-2a_i+19m-9$ adjacent to $v_i$.
\end{enumerate}
Finally, consider a vertex $u$ of status $4A-2B+11m-7$. As stated above, there are exactly $B+3$ descendants of $u$. By assumption (in {\sc 3-Partition}) that  $B/4<a_i<B/2$ for all $i\in \{1,2,\ldots,n\}$, vertex $u$ has exactly three immediate successors/descendants, say $v_i,\ v_j,\ v_k$ where $i,j,k\in \{1,2,\ldots,n\}$, and these vertices have exactly $a_i,a_j,a_k$ descendant leaves, respectively. Thus, $a_i+a_j+a_k=B$, which completes the proof.
\end{proof}

Theorem~\ref{thm:complexity} straightforwardly implies the following corollary.
\begin{coro}
{\sc Tree status recognition} is NP-complete.
\end{coro}
\begin{proof}
In Theorem~\ref{thm:complexity}, there is no loss of generality: when realizing a tree from the constructed sequence of integers, we derived properties 1--5 implying that the only trees possibly realizing the sequence are trees of depth 3.
\end{proof}

Notably, Theorem~\ref{thm:complexity} has also strong algorithmic implications for the following optimization problem.
\begin{problem}[{\sc Minimum Status Correction}] 
Given a sequence of integers $\sigma'$, what is the minimum change on $\sigma'$ (under any norm that measures distance between sequences) that makes $\sigma'$ status realizable in trees? 
\end{problem}
\begin{coro}
Unless $P=NP$, there is no polynomial time constant approximation algorithm for {\sc Minimum Status Correction}.
\end{coro}
\begin{proof}
The proof follows from Theorem~\ref{thm:complexity} by a straightforward gap reduction.
\end{proof}

\section{Polynomial special cases of {\sc Status realizability in trees}}\label{section_realizability_trees}
In this section, we investigate some special cases where {\sc status realizability in trees} can be solved in polynomial time. We also address an extremal case of Problem~\ref{con1}, showing that status injective trees are status unique in trees.

We begin with a result about asymmetric trees. The following result shows that given an integer sequence with $n$ distinct values, it can be determined in polynomial time whether or not this sequence is the status sequence of a tree $T=(V,E)$ of order $n$. Note that such a tree, if realized, is highly asymmetric, as it cannot have any nontrivial automorphisms; see Corollary \ref{obs:treestatusinjectiveautomorphisms} for details.

\begin{theo}\label{answerconj1}
Status injective trees are status unique in trees.
\end{theo}
\begin{proof} The proof is constructive, i.e., we present an algorithm whose input is a sequence $\sigma'$ of $n$ distinct positive integers and the output is either a unique tree $T$ realizing $\sigma'$ or a conclusion that the sequence is not a status sequence of any tree.
\LinesNumbered
\normalem
\begin{algorithm2e}[h]
\label{algorithm_trees}
\textbf{Input:} A sequence of integers $\sigma'=\{a_1,a_2,\ldots,a_n\}$, with $a_1> a_2> \ldots > a_n$.\\
\textbf{Output:} A tree $T$ whose status sequence is $\sigma'$, or an affirmation that such a tree does not exist.\\
\For{$i=1,\ldots, n$}{
Create a vertex $v_i$;\\
$p_i\leftarrow \emptyset$;\hfill $\rhd\ p_i$ is the parent of $v_i$ in $T$\\
$c_i\leftarrow 0$;\hfill $\rhd\ c_i$ is the number of descendants of $v_i$ in $T$\\
}
\For{$i=1,\ldots,n-1$}{
Find $j\in \{i+1,\ldots,n\}$ such that $a_j=a_i-n+2(c_i+1)$\;
\textbf{if} such index $j$ does not exist \textbf{then return} ``$A$ is not the status sequence of a tree" \textbf{else}\\
$p_i\leftarrow v_j$\;
$c_j\leftarrow c_j+c_i+1$\;
}
\Return{$T\leftarrow (\{1,\ldots,n\},\{(v_i,p_i):\ i=1,\ldots,n-1\})$}
\caption{Status injective tree}
\end{algorithm2e}
\ULforem

Algorithm \ref{algorithm_trees} attempts to construct a tree whose vertices $v_1,\ldots,v_n$ correspond to the status values in the input sequence $\sigma'$ (which by assumption are all distinct); the tree is specified by assigning a parent $p_i$ to each vertex $v_i$, except the vertex $v_n$ which is treated as the root. The algorithm also stores and updates the number of descendants of vertex $v_i$ (excluding $v_i$ itself) as the variable $c_i$. In the main loop of the algorithm, the vertices are considered according to decreasing status values.

We will show by induction that after iteration $i$ of the algorithm, the parents of the vertices $v_1,\ldots,v_i$ are uniquely identified as $p_1,\ldots,p_i$, and that the number of descendants of $v_{i+1}$ among $\{v_1,\ldots,v_i\}$ is $c_{i+1}$.

In the first iteration of the algorithm, by Lemma \ref{EJS1976}, vertex $v_1$ (corresponding to status value $a_1$) must be a leaf. By Lemma \ref{lema2.2SL2011}, $a_1-a_{p_1}=(n-(c_1+1))-(c_1+1)$, so $a_{p_1}=a_1-n+2(c_1+1)$. If this value is not in $\sigma'$, then $\sigma'$ is not the status sequence of a tree. Otherwise, the vertex $v_j$ with status value $a_1-n+2(c_1+1)$ is the parent of $v_1$ and vertex $v_1$ becomes a descendant of vertex $p_1$. Thus, after the first iteration, the parent of $v_1$ is uniquely identified as $p_1$, and the number of descendants of $v_2$ among $\{v_1\}$ is $c_2$ (because it is either 0, due the way the variable is initialized (in line 5), or it is 1, i.e., $v_1$, if $p_1=v_2$).

Suppose the statement is true for iterations $1,\ldots,i-1$, and consider iteration $i$. By Lemma \ref{lema2.2SL2011} and by the assumption that the number of descendants of $v_i$ among $\{v_1,\ldots,v_{i-1}\}$ is $c_i$, it follows that $a_i-a_{p_i}=(n-(c_i+1))-(c_i+1)$, so $a_{p_i}=a_i-n+2(c_i+1)$. If this value is not in $\sigma'$, then $\sigma'$ is not the status sequence of a tree (line 8). Otherwise, the vertex $v_j$ corresponding to status value $a_i-n+2(c_i+1)$ is the parent $p_i$ of $v_i$ (line 9) and the descendants of $v_i$, including $v_i$, become descendants of $p_i$ (line 10). See Figure \ref{fig:algo1step} for an illustration. Moreover, since the input of the equation used to compute $p_i$ (in line 7) is uniquely determined by assumption, and the equation is deterministic, it follows that $p_i$ is uniquely identified. By induction, it follows that Algorithm \ref{algorithm_trees} uniquely constructs a tree from a sequence of distinct integers, if such a tree can be constructed. Thus, status injective trees are status unique among trees.
\end{proof}


\begin{figure}[ht]
\centering
  \includegraphics[scale=0.5]{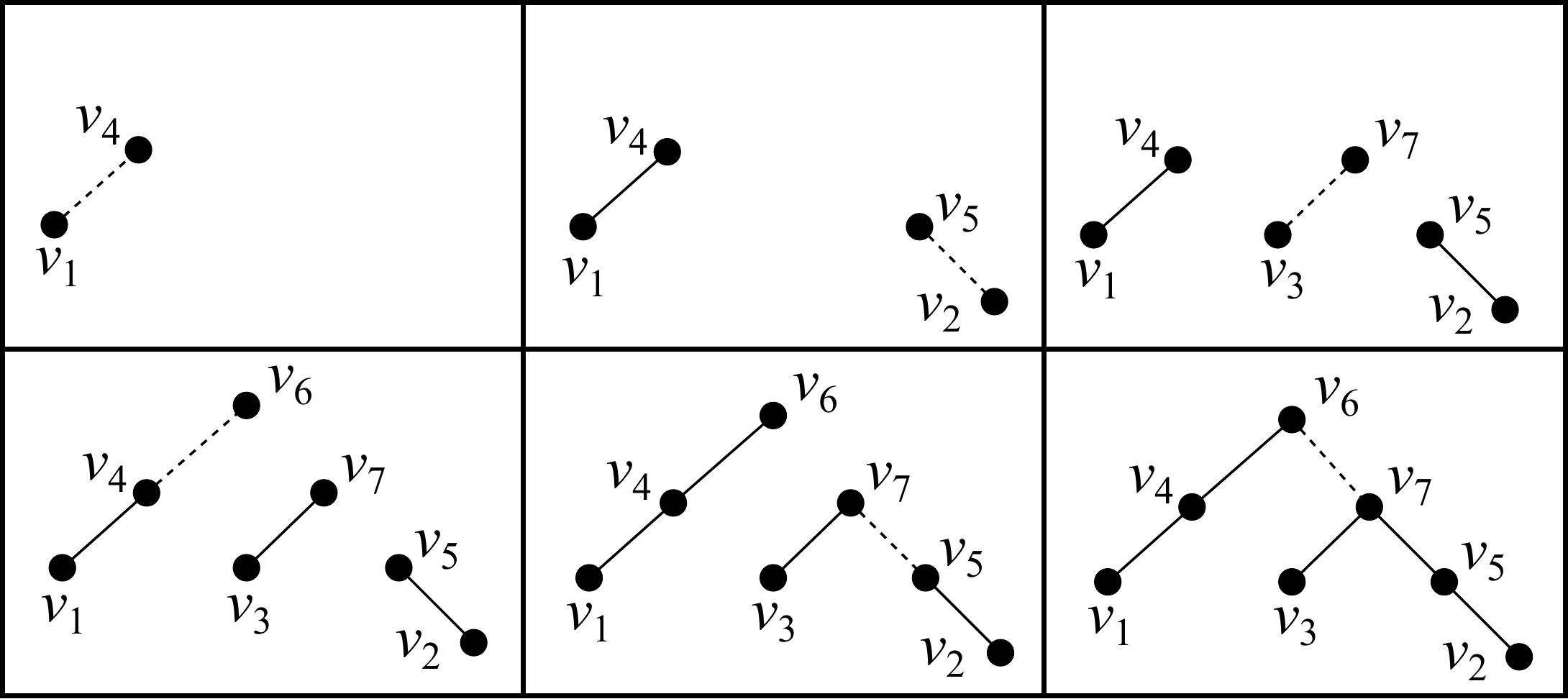}
\caption{Illustration of Algorithm \ref{algorithm_trees} for input $\sigma'=\{19,18,15,14,13,11,10\}$.}
\label{fig:algo1step}
\end{figure}

Note that Algorithm \ref{algorithm_trees} terminates in  $O(n\log n)$ time and can therefore be used constructively and very efficiently.

While Algorithm~\ref{algorithm_trees} is designed to tackle the asymmetric (status injective) cases of {\sc Status realizability in trees}, the remainder of this section addresses polynomially solvable symmetric cases of {\sc SRT-D3}, when the number of distinct status values is bounded by a constant and the degree of the center vertex is bounded by a constant. Note that by Theorem~\ref{thm:complexity}, the problem {\sc SRT-D3} is NP-complete in general, but the reduction requires a high number of distinct status values and a high degree center vertex of the tree. Recall that $k(G)$ denotes the number of distinct status values of a graph $G$. The following result provides a tight lower bound on $k(T)$ for a tree $T$, which will be used in the sequel.
\begin{lema}\label{boundk}
Let $T$ be a tree on $n$ vertices. Then,
\begin{equation*}
k(T)\geq \left\lceil\frac{\text{diam}(T)+1}{2}\right\rceil,
\end{equation*}
and this bound is tight.
\end{lema}
\begin{proof}
Let $\text{diam}(T)$ be the diameter of $T$. It is known that the median of $T$ consists of either a single vertex or two adjacent vertices. Moreover, if the median of $T$ consists of a single vertex, then $\text{rad}(T)=\frac{\text{diam}(T)}{2}$; if the median consists of two vertices, then $\text{rad}(T)=\frac{\text{diam}(T)+1}{2}$. Suppose the median of $T$ consists of a single vertex $v$, and let $w$ be a vertex at maximum distance from $v$. Then, $d(v,w)=\text{rad}(T)=\frac{\text{diam}(T)}{2}$, and by Lemma \ref{EJS1976}, each of the $d(v,w)+1$ vertices on the path between $v$ and $w$ (including $v$ and $w$) has a different status. Thus, $k(T)\geq d(v,w)+1=\frac{\text{diam}(T)}{2}+1\geq \left\lceil\frac{\text{diam}(T)+1}{2}\right\rceil$. Now suppose the median of $T$ consists of two vertices $v$ and $v'$, and let $w$ be a vertex at maximum distance from $v$. Since $v$ and $v'$ are adjacent, the path between $v$ and $w$ passes through $v'$ (otherwise the path from $v'$ to $w$ would be longer). Then, $d(v,w)=\text{rad}(T)=\frac{\text{diam}(T)+1}{2}$, and by Lemma \ref{EJS1976}, each of the $d(v,w)$ vertices on the path between $v$ and $w$, excluding $v$ and including $w$, has a different status. Thus, $k(T)\geq d(v,w)=\frac{\text{diam}(T)+1}{2}$, and since $k(T)$ is an integer, it follows that $k(T)\geq \left\lceil\frac{\text{diam}(T)+1}{2}\right\rceil$.
The bound is tight, e.g., for stars and paths.
\end{proof}

Now, we are ready to address polynomial solvability of {\sc SRT-D3} with a fixed number of distinct status values; these types of trees are highly symmetric.

\begin{theo}
\label{prop3.6}
Let $\sigma'$ be a multiset of $n$ integers and $k$ be the number of distinct values in $\sigma'$. Then, it can be determined whether or not $\sigma'$ is the status sequence of a tree which has depth at most $3$ and degree of the center vertex at most $\delta$ in $2^{O(k^3\delta^3)}n^{O(\delta)}$ time.
\end{theo}

\begin{proof}
Let $r$ be a center vertex. Let us remind that the {\it center} of a tree is the middle vertex or middle two vertices in every longest path. Notice, the center vertex is not necessarily a median vertex. Let $deg(r)\leq \delta$. Since the center vertex $r$ has bounded degree $\delta$, we can enumerate over all possibilities for the set $I$ of vertices at distance $1$ from $r$. There are $n \choose \delta$ of such sets, implying $n^{O(\delta)}$ time for this enumeration. From now on, we assume the set $I$ is given.

In a $\sigma'$-realized tree $T$, consider two vertices, $i\in I$ and an adjacent to it (non-root) vertex having status $j$, see Figure~\ref{fig:ijconfigurations}. For simplicity of the notation we refer to the latter vertex as $j$, though it can be any vertex of that status and there could be many such vertices.
\begin{figure}[ht]
\centering
  \includegraphics[scale=0.4]{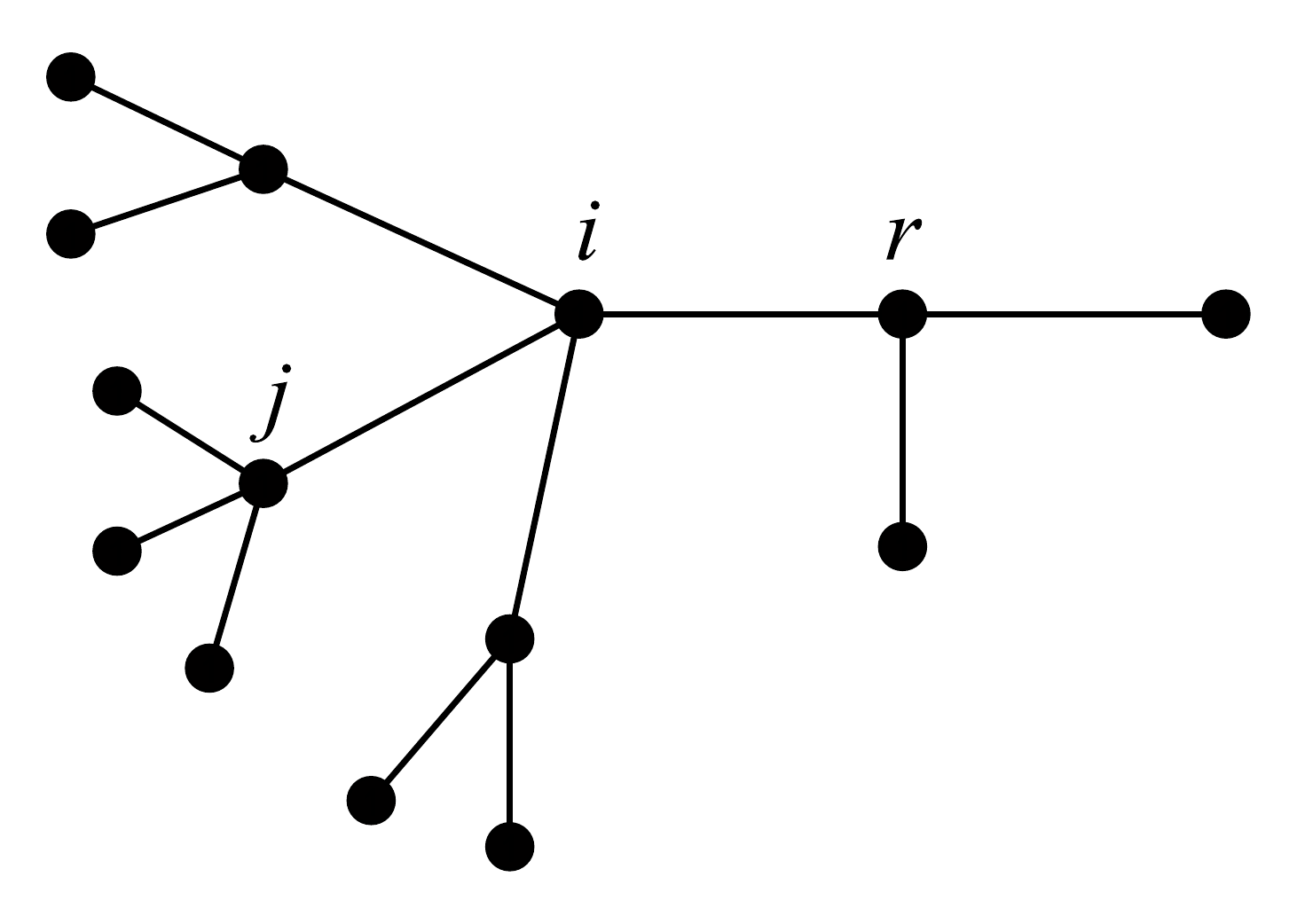}
    \caption{Example of a tree with three $(i,j)$-branches.}
  \label{fig:ijconfigurations}
\end{figure}
We refer to a subtree of $T$ induced by $i\in I$, a vertex of status $j$ and all leaves adjacent to that vertex as {\it $(i,j)$-branch}. Let $a_{\ell ij}$ be the number of status $\ell$ vertices in an $(i,j)$-branch. By Lemma~\ref{lema2.2SL2011}, given the statuses of $i$ and $j$, the parameters $a_{\ell ij}$ are uniquely defined for all $\ell=1,\ldots,k$. Now, let an integer variable $x_{ij}\geq 0$ be the number of $(i,j)$-branches in a $\sigma'$-realized tree $T$ with a given set $I$. The following integer linear program (ILP) solves the problem of finding a $\sigma'$-realized tree $T$ with a given set $I$:  
\begin{equation*}
\begin{array}{ll}
\displaystyle \sum_{i\in I}\sum_{j=1}^k a_{\ell ij}x_{ij}=n_{\ell}  & \forall\ell=1,\ldots,k; \\[5mm]
\displaystyle \sum_{\ell=1}^k\sum_{j=1}^k a_{\ell ij}x_{ij}=\frac{n+s(r)-s(i)-1}{2} & \forall i\in I; \\[5mm]
x_{ij}\in  \mathbb{Z}^{+} & \forall i\in I, j=1,\ldots,k,
\end{array}
\end{equation*}
where $n_\ell$ is the multiplicity of element $\ell$ in $\sigma'\setminus (\{r\}\cup I)$. Here, the first equation preserves the given multiplicities of integers and statuses in $\sigma'$ and $T$, respectively. The second equation guarantees that for all $i\in I$, the necessity condition of Lemma~\ref{lema2.2SL2011} for the edge $(r,i)$ is satisfied. Thus, if the ILP fails to find a feasible solution, the necessity conditions for realizability of $\sigma'$ in a tree with a given set $I$ also fail. On the other hand, given a solution $x_{ij},\ i\in I,\ j=1,\ldots, k$, to the ILP, the straightforward assignment of $x_{ij}$ number of $(i,j)$-branches to $i\in I$ provides a realization of $\sigma'$ in $T$. This implies the correctness of the ILP.        

The number of variables in the ILP is at most $k\delta$. Thus, by Lenstra~\cite{L1983}, the ILP can be solved in $2^{O(k^3\delta^3)}$ time.
\end{proof}

We stress that the assumptions of Theorem \ref{prop3.6} do not restrict the hardness result from Section \ref{section:complexity}. In Theorem \ref{thm:complexity}, the tree also has depth 3, but the number of distinct status values $k$ and the maximum degree $\delta$ of the center vertex are not bounded. This gives rise to the following two open questions. First, is there an intuitive combinatorial algorithm, not relying on the ``black-box'' ILP machinery, that solves {\sc SRT-D3} with $k$ and $\delta$ fixed? Second, if only one of the parameters $k$ or $\delta$ is fixed (not both), is there a polynomial time algorithm for solving {\sc SRT-D3}? Below, we partially address this second question by proving that for $k=2$ and for $k=3$ the problem {\sc Status realizability in trees} can be solved in polynomial time with no additional restrictions. Note that such trees are also highly symmetric.

\begin{defi}
A {\it double star} is a graph that can be obtained by appending $a\geq 1$ pendent vertices to one vertex of $K_2$ and $b\geq 1$ pendent vertices to the other vertex of $K_2$. A {\it balanced double star} is a graph that can be obtained by appending $a\geq 1$ pendent vertices to each vertex of $K_2$.
Let $\mathcal{T}$ be the family of trees which can be obtained by appending $b\geq 1$ pendent vertices to each leaf of a star or to each leaf of a balanced double star.
\end{defi}

Given a graph $G$, the vertices in a set $S\subset V(G)$ are called {\it similar} if for any $u,v\in S$, there exists an automorphism of $G$ that maps $u$ to $v$.

\begin{propo}
\label{prop_doublestar}
Let $T$ be a tree on $n$ vertices. Then $k(T)=2$ if and only if $T$ is a star or a balanced double star.
\end{propo}


\begin{proof}
If $T$ is a star or a balanced double star, then all the leaves of $T$ are similar and therefore have the same status, and all non-leaf vertices are similar have the same status. Moreover, it can easily be verified the status values of a leaf and its non-leaf neighbor are different. Thus, $k(T)=2$.

If $T$ is a tree with $k(T)=2$, then by Lemma \ref{boundk}, $diam(T)\leq 3$. If $diam(T)=0$ or $diam(T)=1$, then $T$ is isomorphic to $K_1$ and $K_2$, respectively, and in both cases $k(T)=1$. If $diam(T)=2$, then $T$ is a star. If $diam(T)=3$, then $T$ is a double star, i.e., $T$ can be obtained by starting from a path $P_4$ with vertices $v_1,v_2,v_3,v_4$ (in path order) and appending $a\geq 0$ pendent vertices to $v_2$ and $b\geq 0$ pendent vertices to $v_3$. Then, $s(v_1)=1+2(a+1)+3(b+1)$, $s(v_2)=1(a+2)+2(b+1)$, and $1(b+2)+2(a+1)=s(v_3)$. If $a\neq b$, then none of these numbers are equal, so  $k(T)>2$. If $a=b$, then $T$ is a balanced double star.
\end{proof}

\begin{theo}
Let $T$ be a tree on $n$ vertices. Then $k(T)=3$ if and only if $T\in \mathcal{T}$.
\end{theo}

%
%
%

\begin{proof}
If $T\in \mathcal{T}$, then all leaves of $T$ are similar, all vertices which are neighbors of leaves are similar, and all vertices which are neither leaves nor neighbors of leaves are similar. Moreover, it can be verified that the status values of these three similarity classes of vertices are different. Thus, $k(T)=3$.

Let $T$ be a tree with $k(T)=3$. By Lemma \ref{boundk}, $diam(T)\leq 5$. If $diam(T)\leq 3$, then $T$ is either $K_1$, or $K_2$, or a star, or a double star. In the first three cases, $T$ has at most two similarity classes. If $T$ is a balanced double star, then by Proposition \ref{prop_doublestar}, $k(T)=2$. If $T$ is a double star which is not balanced, let $v_1$ and $v_2$ be the non-leaf vertices of $T$, let $\ell_1$ be a leaf adjacent to $v_1$, and $\ell_2$ be a leaf adjacent to $v_2$. Then, since $v_1$ and $v_2$ have a different number of leaf neighbors, the status values of $v_1$, $v_2$, $\ell_1$, and $\ell_2$ are all distinct, contradicting the assumption that $k(T)=3$. Thus, $diam(T)=4$ or $diam(T)=5$.

Suppose first that $diam(T)=4$. Let $P$ be a path in $T$ of length 4. Let the vertices of $P$ be $u_1,u_2,r,u_4,u_5$ in path order. Then, because of the diameter restriction, $u_1$ and $u_5$ are leaves of $T$, all neighbors of $u_2$ and $u_4$ besides $u_1$ and $u_5$ (if any) are leaves, and all neighbors of $r$ besides $u_2$ and $u_4$ (if any) are either leaves or vertices whose only neighbors besides $r$ are leaves. Let $t\geq 2$ be the number of non-leaf neighbors of $r$ (including $u_2$ and $u_4$) and let $b\geq 0$ be the number of leaf neighbors of $r$. See Figure \ref{fig_k3} for an illustration.
\begin{center}
\begin{figure}[ht]
\centering
  \includegraphics[scale=0.3]{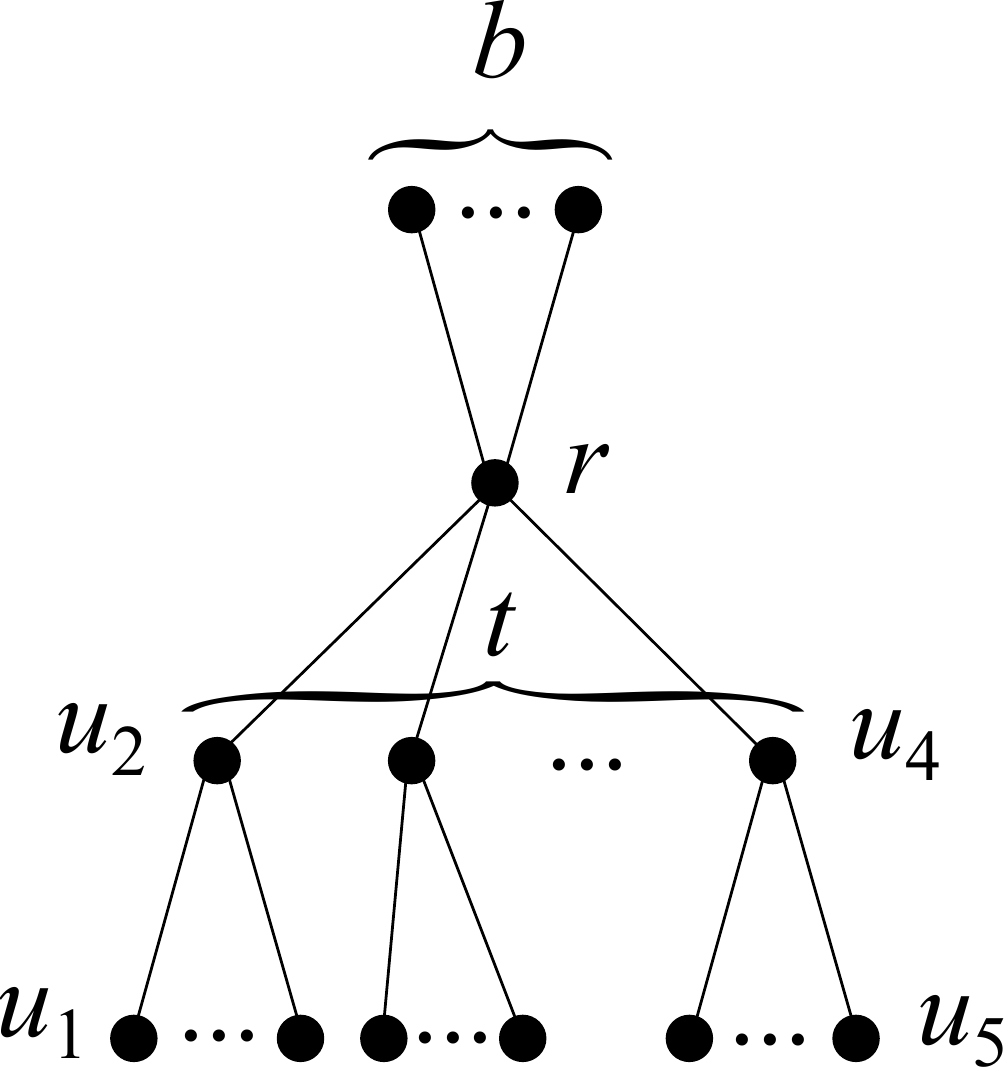}
    \caption{Structure of a tree with 3 distinct status values and diameter 4.}
  \label{fig_k3}
\end{figure}
\end{center}
Suppose first that there exist two non-leaf neighbors of $r$, say $v_1$ and $v_2$, which have different numbers of leaf neighbors. Let $a_1$ be the number of leaf neighbors of $v_1$ and $a_2$ be the number of leaf neighbors of $v_2$. Let $\ell_1$ be one of the leaf neighbors of $v_1$ and $\ell_2$ be one of the leaf neighbors of $v_2$. Let $A$ be the total number of leaves which are adjacent to vertices other than $v_1$, $v_2$, and $r$. Then,
\begin{align*}
s(r)&=1(t+b)+2(a_1+a_2+A)=t+b+2a_1+2a_2+2A\\
s(v_1)&=1(a_1+1)+2(t+b-1)+3(a_2+A)=2t+2b+a_1+3a_2+3A-1\\
s(v_2)&=1(a_2+1)+2(t+b-1)+3(a_1+A)=2t+2b+3a_1+a_2+3A-1\\
s(\ell_1)&=1+2(a_1)+3(t+b-1)+4(a_2+A)=3t+3b+2a_1+4a_2+4A-2\\
s(\ell_2)&=1+2(a_2)+3(t+b-1)+4(a_1+A)=3t+3b+4a_1+2a_2+4A-2.
\end{align*}

Since $a_1\neq a_2$, we have that $s(v_1)\neq s(v_2)$, $s(\ell_1)\neq s(\ell_2)$, $s(v_1)\neq s(\ell_1)$, $s(v_2)\neq s(\ell_2)$, $s(r)\neq s(\ell_1)$, and $s(r)\neq s(\ell_2)$.

Suppose $a_1=t+b+a_2+A-1$. Then, the status values of $r$, $v_2$, $\ell_1$, and $\ell_2$ are all distinct, a contradiction.

Now suppose $a_1\neq t+b+a_2+A-1$. If $a_2\neq t+b+3a_1+A-1$, then the status values of $v_1$, $\ell_1$, $\ell_2$, and $r$ are all distinct, a contradiction. Thus, $a_2= t+b+3a_1+A-1$. Then, the status values of $v_2$, $\ell_1$, $\ell_2$, and $r$ are all distinct, a contradiction.

Thus, all non-leaf neighbors of $r$ must have the same number of leaf neighbors, say $a$. Suppose $r$ has $b\geq 1$ leaf neighbors, and let $w$ be a leaf neighbor of $r$. Let $v$ be any non-leaf neighbor of $r$, and let $\ell$ be any leaf neighbor of $v$. Then,

\begin{align*}
s(\ell)&=1+2a+3(t+b-1)+4(at-a)=3t+3b-2+4at-2a\\
s(w)&=1+2(t+b-1)+3(at)=2t+2b-1+3at\\
s(v)&=1(a+1)+2(t+b-1)+3(at-a)=2t+2b-1+3at-2a\\
s(r)&=1(t+b)+2(at)=t+b+2at.
\end{align*}
Since $t\geq 2$ and $a\geq 1$, it follows that $s(\ell)>s(w)>s(v)>s(r)$. This contradicts the assumption that $k(T)=3$. If $r$ has no leaf neighbors, then $T\in \mathcal{T}$ and the vertices of $T$ have three similarity classes: $r$, the neighbors of $r$, and the leaves of $T$. It can be easily verified that the status values of vertices from these three classes are distinct. Thus, if $k(T)=3$ and $diam(T)=4$, it follows that $T\in \mathcal{T}$. By similar arguments, it can be shown that if $k(T)=3$ and $diam(T)=5$, then $T\in \mathcal{T}$.
\end{proof}

The general case of graphs with status sequences having a unique value, that is, $k(G)=1$, was studied in \cite{GRWC2016}. In that paper the authors obtain tight upper and lower bounds for the unique status value of the status sequence in terms of the number of vertices of the graph.

\section{Equitable, orbit and status partititons}
\label{section_status_partitions}

In this section we explore how the well-known concepts of equitable and orbit partitions relate to the status sequence of a graph.

Denote by $\mathcal{P}=\{V_1,\ldots,V_p\}$ any partition of the vertex set $V$ of a graph $G$. A {\it status partition} of a graph $G$, denoted by $\mathcal{P}_s$, is a partition of $V(G)$ in which any two vertices in the same set have the same status value. An {\it orbit partition} of $G$, denoted $\mathcal{P}_o$, is a partition in which two vertices are in the same set if there is an automorphism which maps $u$ to $v$, for some group of automorphisms of $G$. An {\it equitable} (or {\it regular}) partition of $G$, denoted by $\mathcal{P}_e$, is a partition of $V$ into nonempty parts $V_1,\ldots,V_p$ such that for each $i,j\in \{1,\ldots,p\}$, the number of neighbors in $V_j$ of a given vertex $u$ in $V_i$ is determined by $i$ and $j$, and is independent of the choice of $u$ in $V_i$. More precisely, let the distance matrix $D$ of $G$ be partitioned according to $\mathcal{P}=\{V_1,\ldots,V_p\}$, so that $D_{i,j}$ is a block of $D$ formed by rows in $V_{i}$ and columns in $V_{j}$. The {\it characteristic matrix} $S$ is the $n\times p$ matrix whose $j^{\text{th}}$ column is the characteristic vector of $V_j$, $1\leq j\leq p$. The {\it quotient matrix} of $D$ with respect to the partition $\mathcal{P}$ is the $p\times p$ matrix $B=(b_{i,j})$ whose entries are the average row sums of the blocks of $D$:
\begin{equation*}
b_{i,j}=\frac{1}{|V_{i}|}\1^{\t}D_{i,j}\1=\frac{1}{|V_{i}|}(S^{\t}DS)_{i,j},
\end{equation*}
where $\1$ is the all-one vector. The partition $\mathcal{P}$ is called {\it equitable} if each block $D_{i,j}$ of $D$ has constant row (and column) sum, that is, $SB=DS$. See \cite{G1993} for more details on orbit and equitable partitions.

Let $v$ be a vertex of a graph $G$, and $\epsilon(v)$ denote the greatest distance between $v$ and any other vertex of $G$. The {\it distance partition of $G$ with respect to $v$}, denoted $\mathcal{P}_d(v)$, is a partition of $V$ into parts $\mathcal{P}_d(v)=\{V_0,\ldots,V_{\epsilon(v)}\}$ such that for each $i\in \{1,\ldots,\epsilon(v)\}$, $V_i$ consists of the vertices that are at distance $i$ from $v$.

The following result shows that an orbit partition is just a refinement of a status partition.

\begin{propo}\label{propo:orbitstatus}
Let $G=(V,E)$ be a graph. An orbit partition of $G$ is a status partition of $G$.
\end{propo}
\begin{proof}
Let $\mathcal{P}_o$ be an orbit partition of $G$, and let $u$ and $v$ be two vertices in the same set of $\mathcal{P}_o$. Since there is an automorphism $f:V\longrightarrow V$ which maps $u$ to $v$, it follows that $\text{dist}(u,w)=\text{dist}(v,f(w))$ for all $w\in V$. Since $f$ is a bijection, $s(u)=\sum_{w\in V}dist(u,w)=\sum_{w\in V}dist(v,f(w))=s(v)$. Thus, $\mathcal{P}_o$ is also a status partition of $G$.
\end{proof}

The above elementary result can be applied to the highly asymmetric trees considered in Section \ref{section_realizability_trees}:

\begin{coro}\label{obs:treestatusinjectiveautomorphisms}
A graph with $n$ vertices having status injective sequence cannot have any nontrivial automorphisms.
\end{coro}
\begin{proof}
By Proposition \ref{propo:orbitstatus}, an orbit partition is a status partition, hence if the automorphism group is nontrivial then there is an orbit containing two or more vertices. In this case, the status sequence cannot have $n$ distinct values.
\end{proof}


\begin{obs}\label{obs:equitablepartitionstatuspartition}
Let $G=(V,E)$ be a graph. An equitable partition of $G$ is not a status partition of $G$.
\end{obs}

The following example provides a construction of an infinite family of graphs which illustrates Observation \ref{obs:equitablepartitionstatuspartition}.

\begin{example}\label{ex:Ada}
We will construct an infinite family of graphs whose equitable partitions are not status partitions. For $m\geq 3$, define a graph $G_m$ on vertex set $\{a,b\} \cup A \cup B$ where
\begin{equation*}
A = \{a_{i,j} \ :\ i,j \in \Z_m\} \quad \text{and} \quad B = \{b_{i,j} \ :\ i,j \in \Z_m\}.
\end{equation*}
The vertex $a$ is adjacent to every vertex in $A$ and $b$ is adjacent to every vertex in $B$.
The subgraph of $G_m$ induced by $A$ is a cycle of length $m^2$ with
\begin{equation*}
a_{i,j} \sim
\begin{cases}
a_{i,j+1} & \text{if $j \not = m-1,$ }\\
a_{i+1, 0} & \text{if $j = m-1$}.
\end{cases}
\end{equation*}
The subgraph of $G_m$ induced by $B$ consists of $m$ disjoint cycles of length $m$ with
$b_{i,j} \sim b_{i,j+1}$, for $i, j \in \Z_m$.

We define a matching between the vertices in $A$ and those in $B$ as follows:
\begin{eqnarray*}
&a_{i,j} \sim b_{i,j}  & \quad \text{if $i\neq j$ or $i=m-1$,}\\
&a_{i,i} \sim b_{(i+1) \mod (m-1), (i+1)\mod (m-1)}  & \quad  \text{if $i \neq m-1$.}
\end{eqnarray*}
See Figure \ref{fig_construction} for an illustration.

\begin{figure}[ht]
\begin{tikzpicture}[scale=2]
\draw (-1,0.5) node[anchor=west]{};

\path (0,0.5) coordinate (a);
\path (2,2.25) coordinate (a00);
\path (2.5, 1.5) coordinate (a01);
\path (2.5, 0.5) coordinate (a02);
\path (2.5, -0.5) coordinate (a10);
\path (2.5, -1.5) coordinate (a11);
\path (1.5, -1.5) coordinate (a12);
\path (1.5, -0.5) coordinate (a20);
\path (1.5, 0.5) coordinate (a21);
\path (1.5, 1.5) coordinate (a22);

\path (4, 2.25) coordinate (b00);
\path (4.5, 1.5) coordinate (b01);
\path (3.5, 1.5) coordinate (b02);
\path (4, 0.75) coordinate (b10);
\path (4.5, 0) coordinate (b11);
\path (3.5, 0) coordinate (b12);
\path (4, -0.75 ) coordinate (b20);
\path (4.5, -1.5) coordinate (b21);
\path (3.5, -1.5) coordinate (b22);

\path (6, 0.5) coordinate (b);

\fill (a) circle (1.5pt);
\fill (a00) circle (1.5pt);
\fill (a01) circle (1.5pt);
\fill (a02) circle (1.5pt);
\fill (a10) circle (1.5pt);
\fill (a11) circle (1.5pt);
\fill (a12) circle (1.5pt);
\fill (a20) circle (1.5pt);
\fill (a21) circle (1.5pt);
\fill (a22) circle (1.5pt);

\fill (b) circle (1.5pt);
\fill (b00) circle (1.5pt);
\fill (b01) circle (1.5pt);
\fill (b02) circle (1.5pt);
\fill (b10) circle (1.5pt);
\fill (b11) circle (1.5pt);
\fill (b12) circle (1.5pt);
\fill (b20) circle (1.5pt);
\fill (b21) circle (1.5pt);
\fill (b22) circle (1.5pt);

\draw (a) -- (a00);
\draw (a) -- (a01);
\draw (a) -- (a02);
\draw (a) -- (a10);
\draw (a) -- (a11);
\draw (a) -- (a12);
\draw (a) -- (a20);
\draw (a) -- (a21);
\draw (a) -- (a22);

\draw (b) -- (b00);
\draw (b) -- (b01);
\draw (b) -- (b02);
\draw (b) -- (b10);
\draw (b) -- (b11);
\draw (b) -- (b12);
\draw (b) -- (b20);
\draw (b) -- (b21);
\draw (b) -- (b22);

\draw (a00) -- (a01);
\draw (a01) -- (a02);
\draw (a02) -- (a10);
\draw (a10) -- (a11);
\draw (a11) -- (a12);
\draw (a12) -- (a20);
\draw (a20) -- (a21);
\draw (a21) -- (a22);
\draw (a22) -- (a00);

\draw (b00) -- (b01) -- (b02) -- (b00);
\draw (b10) -- (b11) -- (b12) -- (b10);
\draw (b20) -- (b21) -- (b22) -- (b20);

\draw (a00)--(b11);
\draw (a01) to [out=45, in=135]  (b01);
\draw (a02)--(b02);
\draw (a10)--(b10);
\draw (a11)--(b00);
\draw (a12)--(b12);
\draw (a20)--(b20);
\draw (a21)--(b21);
\draw (a22)--(b22);

\draw (a) node[anchor=east]{\small $a$};
\draw (b) node[anchor=west]{\small $b$};

\draw (a00) node[anchor=south]{\small $a_{0,0}$};
\draw (a01) node[anchor=south]{\small $a_{0,1}$};
\draw (a02) node[anchor=south]{\small $a_{0,2}$};
\draw (a10) node[anchor=south]{\small $a_{1,0}$};
\draw (a11) node[anchor=north]{\small $a_{1,1}$};
\draw (a12) node[anchor=north]{\small $a_{1,2}$};
\draw (a20) node[anchor=south]{\small $a_{2,0}$};
\draw (a21) node[anchor=south]{\small $a_{2,1}$};
\draw (a22) node[anchor=south]{\small $a_{2,2}$};

\draw (b00) node[anchor=south]{\small $b_{0,0}$};
\draw (b01) node[anchor=north]{\small $b_{0,1}$};
\draw (b02) node[anchor=north]{\small $b_{0,2}$};
\draw (b10) node[anchor=south]{\small $b_{1,0}$};
\draw (b11) node[anchor=north]{\small $b_{1,1}$};
\draw (b12) node[anchor=north]{\small $b_{1,2}$};
\draw (b20) node[anchor=south]{\small $b_{2,0}$};
\draw (b21) node[anchor=north]{\small $b_{2,1}$};
\draw (b22) node[anchor=north]{\small $b_{2,2}$};
\end{tikzpicture}
\caption{The graph $G_3$, whose equitable partition is not a status partition.}
\label{fig_construction}
\end{figure}
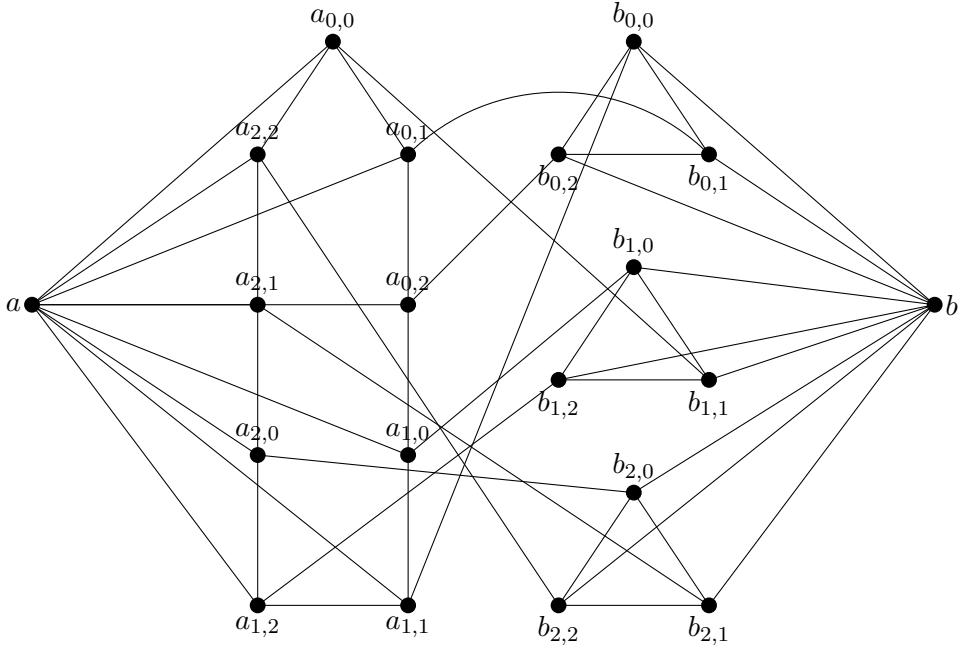

Observe that $\big \{ \{a\}, A, B, \{b\} \big \}$ is an equitable partition of $G_m$.
The vertex $b_{0,0}$ is at distance $1$ from one vertex in $A$, distance $2$ from four vertices in $A$ and distance $3$ from four vertices in $A$.
The vertex $b_{0,1}$ is at distance $1$ from one vertex in $A$, distance $2$ from three vertices in $A$ and distance $3$ from five vertices in $A$. Hence $b_{0,0}$ and $b_{0,1}$ have distinct status values, so $\big \{ \{a\}, A, B, \{b\} \big \}$ is not a status partition of $G_m$.
Note that it can be easily shown that $G_m$ has no non-trivial automorphism.
\end{example}

On the other hand, in the next result we find a necessary condition for a status partition to be an equitable partition.


\begin{propo}\label{suficientconditionstatuspartitionisequitable}
Let $G=(V,E)$ be a graph. A status partition is an equitable partition if and only if for any two vertices $u,v \in V$ with $s(u)=s(v)$, it holds that the distance partitions of $G$ with respect to $u$ and $v$, denoted by $\mathcal{P}_d(u)$ and $\mathcal{P}_d(v)$, have the same quotient matrix.
\end{propo}
\begin{proof}
Let $\mathcal{P}_s=\{V_1, V_2, \ldots,V_k\}$ denote the status partition of $G$, where $k$ is the number of distinct status values in the status sequence of $G$. 

Suppose that the status partition $\mathcal{P}_s$ is an equitable partition. Let $u$ and $v$ be any two vertices with the same status value and such that $u,v\in V_i$ for some $i\in \{1,\ldots,k\}$. Then, since we assume $\mathcal{P}_s$ is an equitable partition, it follows that $u$ and $v$ have the same number of neighbors in each $V_j$ for $j=1,\ldots,k$. Therefore the distance partitions of $G$ with respect to $u$ and $v$ have the same quotient matrix.

Conversely, suppose that for any two vertices $u,v$ with $s(u)=s(v)$, it holds that the partitions $\mathcal{P}_d(u)$ and $\mathcal{P}_d(v)$ have the same quotient matrix. This implies that vertices $u$ and $v$ have the same number of neighbors at each possible distance. Hence it follows that $\mathcal{P}_s$ is an equitable partition.
\end{proof}


Note that the conditions of Proposition \ref{suficientconditionstatuspartitionisequitable} can be verified efficiently. A simple approach to verify that two vertices with the same status value define a distance partition with the same quotient matrix is to build a breadth-first-search tree (in linear time) from each vertex, find the distance partition, and then compute the characteristic matrix and the quotient matrix as described at the beginning of the section.

Finally, we identify a large class of graphs for which a status partition is an equitable partition: distance mean-regular graphs. Distance mean-regular graphs were introduced in \cite{FD2017} as a generalization of both vertex-transitive and distance-regular graphs. Let $G_k(w)$ denote the number of vertices at distance $k$ from a vertex $w$ of a graph $G$. $G=(V,E)$ is a {\it distance mean-regular} graph if, for a given  $u\in V$, the averages of $|G_i(u)\cap G_j(v)|$, computed over all vertices $v$ at a given distance $h\in \{0,1,\ldots,diam(G)\}$ from $u$, do not depend on $u$.  Note that distance-mean regular graphs are {\it super regular graphs}, that is, the number of vertices at a fixed distance is the same for any vertex $v_i\in V$. Note that for a super regular graph, it holds that $s(v_i)$ is constant for any $v_i\in V$.

\begin{coro}
If $G$ is a distance-mean regular graph, then every status partition is an equitable partition.
\end{coro}
\begin{proof}
If $G$ is a distance-mean regular graph, then by Proposition 2.2 in \cite{FD2017}, the quotient
matrix $B$ with respect to the distance partition of every vertex is
the same. This implies that the condition of Proposition \ref{suficientconditionstatuspartitionisequitable} is satisfied and therefore it follows that every status partition in $G$ is an equitable partition.
\end{proof}

%

We conclude with two open problems. The first is related to status partitions.
\begin{problem}\label{prob_orbits_status}
For which graphs is the status partition an orbit partition of some automorphism group?
\end{problem}

Note that the property in Problem \ref{prob_orbits_status} holds for vertex-transitive graphs. The second open problem is related to status injective trees. 

\begin{problem}\label{prob_status_inj}
For which families of graphs $\mathcal{F}$ does it hold that status injective trees are status unique in $\mathcal{F}$?
\end{problem}

\section*{Acknowledgements}
  We would like to thank Ada Chan for helpful discussions and for finding Example \ref{ex:Ada}.

\end{document}